\documentclass[11pt,letterpaper]{amsart}%{article}
\pdfoutput=1 
\usepackage{etex}
\usepackage[latin1]{inputenc}
\usepackage{multirow}
\usepackage{amsmath}
\usepackage{subfigure}
\usepackage{amsfonts}
\usepackage{amsthm}
\usepackage{array,nicefrac}
\usepackage{amsthm}
\usepackage{tikz}	
\usetikzlibrary{arrows}
\usetikzlibrary{patterns} 
\usepackage{multirow}
\usepackage{amssymb}
\usepackage{amscd, amsfonts}
\usepackage[dvips]{epsfig}
\usepackage{verbatim}
\usepackage[usenames,dvipsnames]{pstricks}
\usepackage{epsfig}
\usepackage{pst-grad} % For gradients
\usepackage{pst-plot} % For axes
\usepackage{pdftricks}
\usepackage{pstricks}
\usepackage{epsf}
\usepackage{multirow}
\usepackage{amssymb}
\usepackage{verbatim} 
\usepackage{graphicx}
\usepackage[all]{xy}
\usepackage{lscape}
\usepackage{hyperref}

\newtheorem{thm}[equation]{Theorem}

\newtheorem{lemma}[subsection]{Lemma}
\numberwithin{equation}{section}

\newtheorem{defn}[subsection]{Definition}

\theoremstyle{definition}
\newtheorem{rmk}[subsection]{Remark}
\newtheorem{ex}[subsection]{Example}

\begin{document}
\title{A note on graded hypersurface singularities}
\date{\today}
\author[Gallardo]{Patricio Gallardo}
\address{Department of Mathematics\\Stony Brook University\\ Stony Brook, NY 11794}
\email{pgallardo@math.sunysb.edu}
\maketitle{}
\begin{abstract}
For a weighted quasihomogeneous two dimensional hypersurface  singularity, we define a smoothing with unipotent monodromy and an isolated graded normal singularity. We study the natural weighted blow up of both the smoothing and the surface.  In particular, we describe our construction for the quasihomogeneous singularities of type I, the 14 unimodal exceptional singularities  and we relate it to their stable replacement.
\end{abstract}

\section{Introduction}
 
We are interested in generalizing the procedure for calculating the stable replacement of a plane curve  singularity  to the case of an isolated  surface singularity $S_0 \subset \mathbb{C}^3$ with a good $\mathbb{C}^*$-action. In the one dimensional case, the stable replacement for a generic smoothing involves a base change $t^m \to t$ follow it by a weighted blow up. This construction was studied by Hassett  \cite{hassett} for the case of toric and quasitoric plane curve singularities (see Remark \ref{hst}).  The purpose of this note is to discuss the generalization of that setting for the 14 exceptional unimodal surface singularities and the quasihomogeneous singularities of type I.  

Our approach uses the natural grading associated to a surface singularity $S_0$ with a $\mathbb{C}^*$-action and its weighted blow up, also known as the Steifer partial resolution.    Indeed, we consider two perspectives for constructing our one dimensional families of surfaces: First, we impose  the conditions that the smoothing $ X \to \Delta$ of $S_0$ is unipotent and that it has a graded isolated hypersurface singularity.  The former condition is of particular interest because the monodromy theorem implies semistable families have unipotent monodromy  (see \cite[pg. 106]{morrison1984clemens}).   The second type of smoothings are constructed by  considering the possible ways  the orbifolds singularities associated to the minimal resolution of $S_0$  can be extended to isolated threefold canonical singularities. 

The main results of this note are Theorem \ref{thm1} which considers surface singularities of type I; and Theorem \ref{thm2} which considers the 14 exceptional unimodal singularities.  In particular, the relationship between unimodal singularities and K3 surfaces is  analogous to the relationship between  the cusp singularity on a plane curve and the elliptic tail that appears in the context of moduli of curves.    In work in progress, we are generalizing this picture for minimal elliptic singularities which is a significant larger family of non log canonical singularities; these singularities appear naturally on moduli constructions of surface of general type
 \cite[Prop. 4.12]{gallardo2013git}.   On Section \ref{hm}, we provide examples  illustrating their behaviour which is  different to the unimodal case.  

The idea of comparing the partial resolution of a smoothing with that of an appropriate hypersurface section has been used recently in a more general setting by Wahl \cite{wahl2011log}. The relationship between these K3 surfaces, the unimodal singularities and their unipotent smoothings appears in several instances (see \cite{dolgachev1996mirror},  \cite{geoK3},  \cite{prokhorov}). We must warn our readers that developments from the work of Pinkham \cite{pinkham1978deformations}  and Wahl \cite{wahl1988deformations} are not included in this note. 
Moreover,  Theorem \ref{thm2} appears more explicitly in the recent work of J.  Rana \cite{Rana} in Fuchsian singularities where it is conjectured and partially proved; her work overlap with ours but her methods are different.  Our applications to moduli problems will appear somewhere else. 

\subsection{Acknowledgement} 
I am grateful to my advisor, Radu Laza, for his guidance and help.  I have benefited from discussions with J. Tevelev, J. Rana, K. Ascher, B. Hassett, N. Durgin and I. Dolgachev.  I am also grateful to the CIE and the Mathematics department at Stony Brook University for their support.  Our presentation follows closely that of Wahl \cite{wahl2011log}  and Hassett  \cite{hassett} from which I benefited greatly.

%%%%%%%%%%%%%%%%
\section{Preliminaries}

\subsection{On weighted quasihomogeneous hypersurfaces}

Let $S_0$ be the germ of a surface singularity with a $\mathbb{C}^*$-action defined by 
$
\sigma (t, (x_1, \ldots, x_n)) = (t^{w_0}x_1, \ldots, t^{w_n}x_n)
$. 
If the integers $w_i  >0$, then we say that $\sigma$ is a good $\mathbb{C}^*$-action.  We will always suppose the action is good and our singularities are isolated. In particular, the case of  two dimensional hypersurface singularities were classified by Orlik, Wagreich \cite[3.1]{orlik1971isolated} and Arnold \cite{arnold}. Indeed, any of those singularities have a topologically trivial deformation into one of the following polynomials 
\cite[Cor. 3.6]{xu1989classification}:
\begin{align*}
 \text{Class I}  & = x^{p_0}+y^{p_1}+z^{p_2}  &   & 
 \text{Class II} = x^{p_0}+y^{p_1}+yz^{p_2}
\\
\text{Class III} &=   x^{p_0}+y^{p_1}z+yz^{p_2}  & &  
\text{Class IV}  =  x^{p_0}+y^{p_1}z+xz^{p_2}
\\
\text{Class V} &= x^{p_0}y+y^{p_1}z+xz^{p_2}  & & \\
\text{Class VI} &=x^{p_0}+xy^{p_1}+xz^{p_2}  +y^az^b  & &  ({p_0}-1)({p_0}a+{p_2} b)={p_0}{p_1}{p_2}   \\
\text{Class VII} &=x^{p_0}y+xy^{p_1}+xz^{p_2}  +xz^a+y^az^b & &  ({p_0}-1)({p_0}a+{p_2} b) =b({p_0}{p_1}-1) 
\end{align*}
where $p_i \geq 2$. In this paper, we suppose the equation defining $S_0$ belongs to one of the previous classes (see Remark \ref{notunf}) and that $S_0$ has good properties (see Remark \ref{gd}).  For us $ \tilde{w}=(w_0,w_1,w_2)$ denotes the unique set of integer weights for which the surface singularity is quasihomogeneous,  $gcd(w_0,w_1,w_2)=1$, and the weighted multiplicity $\deg_{\tilde{w}}(S_0)$ reaches its smallest integer value. 
\begin{rmk}
In an abuse of notation we will not distinguish between the germ of a hypersurface singularity and a generic hypersurface of large degree with such a singularity.   In general, we take  
$S_0 \subset \mathbb{C}^3$ to be a surface singularity while $X \subset \mathbb{C}^4$ denotes a threefold one. 
\end{rmk}
\begin{rmk}\label{notunf}
There is not a uniform notation among the references.  For example, the class V in \cite{arnold} is labelled as  IV in \cite{xu1989classification}. We used the notation from \cite{xu1989classification}.
\end{rmk}
\begin{ex}\label{wgt}
For singularities of class I, the explicit expressions for the weights and their weighed multiplicity  are:
\begin{align*}\label{w}
\tilde{w}
& =
\left(
\frac{p_1p_2}{gcd(p_0p_1,p_1p_2,p_0p_2)},
\frac{p_0p_2}{gcd(p_0p_1,p_1p_2,p_0p_2)}, 
\frac{p_0p_1}{gcd(p_0p_1,p_1p_2,p_0p_2)}
\right)
\\
deg_w(f)
&=  
\frac{p_0p_1p_2}{gcd(p_0p_1,p_1p_2,p_0p_2)}.
\end{align*}
%%%%%%%%%%%%%%%%%%%%
The expressions for the other classes are increasingly cumbersome, so we will not display them here.
\end{ex}
By the work of Orlik and Wagreich, it is known that the dual graph of the minimal resolution of $S_0$ is star shaped.  The central curve is a curve whose genus is determined by the weights $\tilde{w}$  while the other curves in the branches of the minimal resolution are rational  (see \cite[Thm 2.3.1, Thm 3.5.1]{orlik1971isolated}).  The contraction of these branches generates a set of cyclic quotient singularities naturally associated to $S_0$.
\begin{defn} 
Let $S_0$ be a surface singularity with a $\mathbb{C}^*$-action, we denote by $\mathcal{B}_{S_0}$  the set of cyclic quotient singularities supported on the central curve  $(S_0 \setminus 0) / \mathbb{C}^*$ and obtained by contracting the branches  on the minimal resolution of $S_0$.
\end{defn}
%%%%%%%%%%%%%%%
\begin{rmk} \label{dol}
Let $S_0$ be a normal Gorenstein surface singularity with a good $\mathbb{C}^*$- action. Its coordinate ring is a graded algebra 
$
A= \bigoplus_{k=0} ^{\infty} A_k
$ 
where $A_0 = \mathbb{C}$ and the singularity is defined by its maximal ideal.  Following Dolgachev, there exist a simply connected Riemann surface $C$, a discrete cocompact group 
$\Gamma \subset Aut(C)$ and an appropriate line bundle $\mathcal{L}$ such that 
$A_k = H^0(C, \mathcal{L}^k)^{\Gamma}$.   Moreover, there is a divisor $D_0 \subset C$ and points $p_i \in C$ such that: 
$$  
A_k = L \left( kD_0 + \sum_{i=1}^{r} \Big\lfloor k \frac{\alpha_i-\beta_i}{\alpha_i} \Big\rfloor \right) 
$$
where $L(D_R)$ is the space of meromorphic functions with poles bounded by $D_R$. The set of numbers $ b:=\deg(D_0)+r$, $(\alpha_i,\beta_i)$ and $g(C)$ are known  as the orbit invariants of $S_0$.  Furthermore, there is a number $R$, known as the exponent of $S_0$ that satisfies the relationships: 
\begin{align*}
R \left( -b+ \sum_{i=1}^{r} \frac{\beta_i}{\alpha_i}  \right)
=
\left( - \deg(K_C) -r + \sum_{i=1}^{r} \frac{1}{\alpha_i} \right)
&; &
R \beta_i \equiv 1 \mod (\alpha_i) \;\; \forall i=1, \ldots r
\end{align*}
with this notation the cyclic singularities are
$\mathcal{B}_{S_0} =\left\{ \frac{1}{\alpha_i}(1,\beta_i)  \right\}$. 
\end{rmk}
In general there are  explicit formulas for finding the singularities in $\mathcal{B}_{S_0}$ 
(see \cite[Sec 3.3]{orlik1971isolated}).  Next, we describe two examples that will be relevant in the following sections.
\begin{ex}\label{bf}
Let $S_0$ be the singularity induced by $f(x,y,z)=x^{p_0}+y^{p_1}+z^{p_2}$  with 
$p_i \geq 2$.  Then, the weights $w_i$ are as on  Example \ref{wgt}, $\alpha_k=gcd(w_i,w_j)$, 
$R=\deg_{\tilde{w}}(f)-\sum_i w_i$ and  $\mathcal{B}_{S_0} $ has 
cyclic quotient singularities of type
$
\frac{1}{\alpha_k}(1,\beta_k)
$ 
where $w_k\beta_k \equiv -1 \mod gcd(w_j,w_i) $.  In the minimal resolution of $S_0$, let 
$\hat{C}$ be the proper transform of $(S_0 \setminus 0) / \mathbb{C}^*$.  
For example, consider the $E_{20}$ singularity which associated equation is $x^{2}+y^3+z^{11}$ then
\begin{align*}
\mathcal{B}_{E_{20}} = \left\{
\frac{1}{2}(1,1), \frac{1}{3}(1,2), \frac{1}{11}(1,9)
\right\}
\end{align*}
For all the classes of  quasihomogeneous singularities, the self intersection of $\hat{C}$   is given by
(see \cite[Thm 3.6.1]{orlik1971isolated}) 
$$
-\hat{C}^2 =\frac{ \deg_{\tilde{w}}(f)}{w_0w_1w_2} + \sum_{k=1}^{r} \frac{\beta_k}{\alpha_k}
$$
although in general $\alpha_k$ may be different to $gcd(w_j,w_i) $.  
\end{ex}
\begin{defn} (notation as on Remark \ref{dol}) A surface singularity $S_0$ is called Fuchsian if $C$ is the upper half plane and $\mathcal{L}=K_C$.  In particular, the 14 exceptional unimodal singularities are Fuchsian ones. 
\end{defn}
\begin{ex}\label{fsh}
Let $S_0$ be a Fuchsian singularity, then 
$$
\mathcal{B}_{S_0} = \left\{  \frac{1}{\alpha_i}(1,1)   \right\}
$$
where $\alpha_i$ is an orbit invariant as on Remark \ref{dol} (for a list of the possible $\alpha_i$ see \cite[Sec 1, Table 1]{ebeling2003poincare}).
\end{ex}
\begin{defn} Given a cyclic quotient singularity $T_1:=\frac{1}{r}(1,d_1) $. We say that 
$ \frac{1}{r}(1,d_2)$ is  dual to $T_1$ if the isolated threefold cyclic quotient singularity 
$$
\frac{1}{r}(1,d_1,d_2)
$$
is canonical. Let $\mathcal{B}_{S_0} = \{ T_1, \ldots T_n \} $ be the set of cyclic quotient singularities associated to  $S_0$.  We say  the set of cyclic quotient singularities
$\mathcal{\hat{B}}_{S_0} = \{ \hat{T}_1, \ldots , \hat{T}_n \} $ is dual to 
$ \mathcal{B}_{S_0}  $ if $\hat{T}_i$ is dual to $T_i$ for every $i$.
\end{defn}
\begin{rmk}\label{exdual} We use the Reid-Tate criterion \cite[pg 105]{smmp} and the classification due to Morrison and Stevens
\cite[Thm 2.4]{morrison1984terminal},  \cite[Thm 3]{morrison1985canonical}
for finding the dual sets  $\mathcal{\hat{B}}_{S_0}$.  The isolated singularity 
$\frac{1}{r}(1,d_2) $ is the dual $\frac{1}{r}(1,d_1)$ if and only if one of the following equations holds (the three first imply the threefold singularity is terminal, the last one implies it is Gorenstein):
\begin{align*}
1+d_1 \equiv 0 \mod ( r ) &; & 1+d_2 \equiv 0 \mod ( r ) &; &d_1+d_2 \equiv 0 \mod ( r ) 
&; & d_1+d_2 \equiv -1 \mod ( r ) 
\end{align*}
if $r=7,9$ we also must also consider two, somehow exceptional, canonical cyclic singularities:
\begin{align*}
\frac{1}{9}(2,8,14) & & \frac{1}{7}(1,9,11)
\end{align*}
\end{rmk}\label{dfc}
\begin{ex} Let $S_0$ be a Fuchsian singularity as on Example \ref{fsh}, then we have the dual sets $\mathcal{\hat{B}}_{S_0}= \left\{  T_i   \right\}$ where  $T_i$ is either 
$
\frac{1}{\alpha_i}(1,\alpha_i-1)
$
or
$
\frac{1}{\alpha_i}(1,\alpha_i-2)
$.
\end{ex}
\begin{ex} \label{notunq}
Let $\mathcal{B}_{E_{20}}$ be as on Example \ref{bf}. Then, two possible dual sets are: 
\begin{align*}
(\hat{\mathcal{B}}_{E_{20}})_{1} = \left\{
\frac{1}{2}(1,1), \frac{1}{3}(1,1), \frac{1}{11}(1,2)
\right\}
& &
(\hat{\mathcal{B}}_{E_{20}})_{2} = \left\{
\frac{1}{2}(1,1), \frac{1}{3}(1,1), \frac{1}{11}(1,10)
\right\}
\end{align*}
\end{ex}
%%%%%%%
\section{On the smoothing of surface singularities} \label{wbup}

In this section, we construct two types of one dimensional families of surfaces motivated by a similar construction due to Hassett \cite{hassett} in the context of plane curves singularities (see Remark \ref{hst}). Our definitions spring from two questions:
\begin{enumerate}
\item  What is the base change need it for calculating the stable replacement of a non log canonical surface singularity? 
\item Given the union of two surfaces with singularities $\mathcal{B}_{S_0}$ and $\mathcal{\hat{B}}_{S_0}$, is it possible to find a $\mathbb{Q}$-Gorenstein smoothing of them?
\end{enumerate}
Our partial advances in these questions yield the  Theorems \ref{thm1} and  \ref{thm2}.  Next, we describe  the weighted blow ups that are used  in our work.

Let $X:=\operatorname{Spec}(A)$ be an isolated graded normal singularity with a good $\mathbb{C}^*$-action.  
Then, $A$ can be written as the quotient of a graded polynomial ring $\mathbb{C}[x_1, \ldots x_n]$ where the variables $x_i$ have weights $w_i >0$ (\cite[1.1]{orlik1971isolated}). 
The $\mathbb{C}^*$-action determines a weighted filtration whose blow up is known as the Steifer partial resolution of $X$  (for more details see \cite[Sec. 1]{wahl2011log}). This resolution depends of the grading, and in our case it is induced by  the weighted blow up $\pi_{\mathbf{w}}$ of $\mathbb{C}^n$ with respect the weights $\textbf{w}$ which we will denoted as $Bl_{\mathbf{w}}\mathbb{C}^n$.  In this case, $\pi_{\mathbf{w}}$  amounts to blowing up  $\mathbb{C}^n$ along the ideal
$$
\left(x_1^{d/{w_1}}, \ldots,x_n^{d/{w_n}}  \right)
$$
where $d=lcm(w_0, \ldots, w_n)$. The exceptional divisor associated to $\pi_{\mathbf{w}}$ is 
$E_{\mathbf{w}} := \mathbb{P}(w_0, \ldots ,w_n)$ and the exceptional divisor on the proper transform $\tilde{X}$ of $X$ is induced by $\tilde{X} \cap E_{\mathbf{w}}$.
 On our applications,  $X \subset \mathbb{C}^4$ is a smoothing of $S_0$ with an isolated singularity that has a $\mathbb{C}^*$-action.   By restricting
$\pi_{\mathbf{w}}$ to its central fiber we have a weighted blow up of $ S_0 \subset \mathbb{C}^3$ with weights $\tilde{w}$ and exceptional divisor  $E_{\tilde{w}} = \mathbb{P}(w_0,w_1,w_2)$. The final configuration is given by (see Figure \ref{big}):
\begin{displaymath}
\xymatrix{
C:=S_1 \cap E_{\tilde{w}} \subset  S_1 +E_{\tilde{w}} \subset Bl_{\tilde{w}}\hat{\mathbb{C}}^3
 \ar@{->}[r] \ar[d]_{\pi_{\tilde{w}}} 
& \tilde{X}  + E_{\pi_{\mathbf{w}}}  \subset  
Bl_{\mathbf{w}}\hat{\mathbb{C}}^4 \ar[d]^{\pi_{\mathbf{w}}} 
\\
S_0 \subset \hat{\mathbb{C}}^3
\ar@{->}[r]  & 
 X  \subset \hat{\mathbb{C}}^4 
}
\end{displaymath}

The central fiber $\tilde{X}|_0$  decomposes as the union of  two surfaces $S_1$ and $S_T$. The surface $S_1$ is the proper transform of $S_0$ and $S_T \subset \mathbb{P}(w_0,w_1,w_2,w_3)$ is the exceptional surface contained in the proper transform of $\tilde{X}$.  
The exceptional divisor $E_{\tilde{w}}$ is not contained in the threefold $\tilde{X}$, but rather it intersects $\tilde{X}$ along the exceptional curve $C \subset S_1$ which satisfies 
$$
C:= S_1 \cap E_{\pi_{\tilde{w}}} = S_1 \cap S_T.
$$
The singularities on $S_1$ are the ones on 
$\mathcal{B}_{S_0}$.  The singularities on $S_T$ depend of the weights $\mathbf{w}$. Note that $\tilde{X}$ may have non isolated singularities. Our applications spring from the fact that, for the appropriate cases, the new one dimensional family of surfaces $\tilde{X} \to \Delta$ have  at worst canonical singularities.  
\begin{figure}[h!]
\begin{center}
\begin{tikzpicture}[line cap=round,line join=round, >=triangle 45, x=1cm,y=1cm, scale=6]
\baselineskip=5pt
\hsize=6.3truein
\vsize=8.7truein
\tikzstyle{every node}=[font=\small]
\definecolor{b}{rgb}{0.6,0.2,0}
\definecolor{black}{rgb}{0.25,0.25,0.25}
%%% end defintions
\draw [color=b]   (0,0)-- (0.2,0.1);
\draw [color=b ](0.2,0.1)-- (0.1,0.26);
\draw  [color=b]  (0.1,0.26)-- (0,0);
\fill [color=red] (0.1,0.26) circle (0.3pt);
%\path [b, fill,very nearly transparent ] 
\pattern[pattern=vertical lines, pattern color=b, nearly transparent]  (0,0)-- (0.2,0.1) --(0.1,0.26);  
\draw [color=b] (0,0.4)-- (0.2,0.5);
\draw [color=b] (0.2,0.5)-- (0.2,0.7);
\draw [color=b] (0,0.6)-- (0,0.4);
%\path [b, fill,very nearly transparent ] 
\pattern[pattern=vertical lines, nearly transparent, pattern color=b]   (0,0.6)--(0,0.4)-- (0.2,0.5) --(0.2,0.7);  
\draw [blue] (0.2,0.7)-- (0.2,0.9);
\draw [blue] (0.2,0.9)-- (0,0.8);
\draw [blue] (0,0.8)-- (0,0.6);
%\path [blue, fill,very nearly transparent ] 
\pattern[pattern=dots, nearly transparent, pattern color=blue] (0.2,0.7)-- (0.2,0.9)-- (0,0.8) -- (0,0.6);  
\draw [color=red] (0,0.6)-- (0.2,0.7);
\filldraw[red] (0.066,0.632) circle (0.2pt);
\filldraw[red] (0.133,0.665) circle (0.2pt);
%\node at (0.1,0.55) {$\mathcal{B}_{S_0}$ };
\node at (0.1,0.69) { $C$};
\node at (-0.05,0.50) { $S_{1}$};
\draw [->] (0.1,0.38) -- (0.1,0.31); \node at (0.17,0.35) { $\pi_{\tilde{w}}$};
\draw [->] (1.2,0.38) -- (1.2,0.31); \node at (1.3,0.35) { $\pi_{\textbf{w}}$};
\draw [->] (0.3,0.4) -- (0.7,0.4);
\node at (-0.05,0.75) { $E_{\pi_{\tilde{w}}}$};
\node at (-0.05,0.1) { $S_{0}$};
%%
% smoothing
\begin{scope}[shift={(0.5,0)}]
\draw (0.4,0)-- (0.4,0.2); %A
\draw (0.4,0.2 )-- (0.6,0.3); %B
\draw [dashed](0.4,0) --(0.6,0.1);
\draw [dashed] (0.6,0.1) -- (0.6,0.3);
%\draw [dashed] 
\draw (0.8,0)-- (0.8,0.2); %B
\draw (0.8,0.2)-- (1,0.3);
\draw (0.8,0)-- (1,0.1);
\draw (1,0.3)-- (1,0.1);
%% triangle in the middle
\draw [color=b] (0.6,0.01)-- (0.8,0.11);
\draw [color=b] (0.8,0.11)-- (0.7,0.27);
\draw [color=b] (0.7,0.27)-- (0.6,0.01);
%\path [b, fill,very nearly transparent ]
 \pattern[pattern=vertical lines, ,pattern color=b, nearly transparent]  (0.6,0.01)-- (0.8,0.11)-- (0.7,0.27) -- (0.6,0.01);  
\fill [color=red] (0.7,0.27) circle (0.3pt);
\node at (0.35,0.1) { $X$};
\node at (0.35,0.5) { $\tilde{X}$};
\node at (0.9,0.7) { $S_{T}$};
\node at (0.7,0.55) { $S_{1}$};
\node at (0.6,0.9) { $E_{\pi_{\textbf{w}}}$};
%%%
% arcs
%
\draw (0.4,0.2) .. controls (0.6,0.21) .. (0.8,0.2);
\draw (0.4,0) .. controls (0.6,0.01) .. (0.8,0);
\draw [dashed] (0.6,0.1) .. controls (0.8,0.11) .. (1,0.1);
\draw (0.6,0.3) .. controls (0.8,0.31) .. (1,0.3);
\draw (0.4,0.4)-- (0.4,0.6)-- (0.6,0.7);
\draw (0.6,0.7)-- (0.61,0.7);
\draw [dashed] (0.61,0.7)-- (0.82,0.7);
\draw [dashed] (0.6,0.7) --(0.6,0.5);
\draw [dashed](0.4,0.4)-- (0.6,0.5);
\draw [dashed] (0.6,0.5) .. controls (0.6,0.51).. (1,0.5);
\draw (0.8,0.4)-- (0.8,0.7);
\draw (0.8,0.7)-- (1,0.8);
\draw (1,0.8)-- (1,0.5);
\draw (1,0.5)-- (0.8,0.4);
\draw (0.4,0.4) ..controls (0.6,0.4) .. (0.8,0.4);
\draw [blue, dashed] (0.81,0.7) -- (0.81,0.9);
\draw [blue] (0.61,0.6)- - (0.61,0.8);
\draw [b] (0.61,0.4)- - (0.61,0.6);
\draw [blue] (0.61,0.8)- - (0.81,0.9);
\draw [blue](0.61,0.8)-- (0.71,0.8);
\draw [blue](0.71,0.8)-- (0.71,0.7);
\draw [blue, dashed] (0.81, 0.7) -- (0.91,0.8); %ST
%\path [blue, fill,very nearly transparent ] 
\pattern[ pattern=dots, nearly transparent, pattern color=blue]  (0.61,0.6)-- (0.81, 0.7) -- (0.81,0.9)-- (0.61,0.8) ;
\draw [blue](0.61,0.6)-- (0.71,0.7); %ST
\draw [blue] (0.71,0.7)-- (0.91,0.8);%ST
%\path [red, fill,very nearly transparent ]   
\pattern[pattern=crosshatch, nearly transparent, pattern color=red] (0.61,0.6) -- (0.81,0.7)--  (0.91,0.8) --(0.71,0.7);
\draw [blue](0.71,0.8)-- (0.91,0.9); 
\draw [blue](0.81,0.9)-- (0.91,0.9);
\draw [blue](0.91,0.9)-- (0.91,0.8);
\draw [dashed, red] (0.61,0.6)-- (0.81,0.7); 
\draw (0.91,0.8)-- (1,0.8); %x
\draw (0.4,0.6)-- (0.61,0.6); %x
%\draw (0.91,0.8)-- (0.82,0.7);
\draw [b,dashed](0.81,0.7)-- (0.81,0.5); %s1
\draw [b,dashed ](0.81,0.5)-- (0.61,0.4); %s1
\pattern[pattern=vertical lines, nearly transparent, pattern color=b]   (0.61,0.6) -- (0.81,0.7) -- (0.81,0.5) -- (0.61,0.4);
%\path [b, fill,very nearly transparent ]  (0.61,0.6)-- (0.81,0.7)-- (0.81,0.5) -- (0.61,0.4);
\filldraw[red] (0.74,0.665) circle (0.2pt);
\filldraw [red] (0.68,0.632) circle (0.2pt);
%useful commnads: http://stuff.mit.edu/afs/athena/contrib/tex-contrib/beamer/pgf-1.01/doc/generic/pgf/version-for-tex4ht/en/pgfmanualse10.html
\end{scope}
\end{tikzpicture}
\caption{Weighted blow up setting}\label{big}
\end{center}
\end{figure}
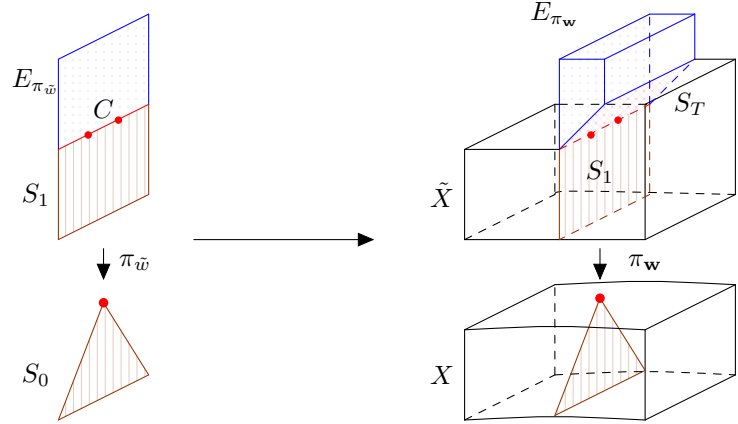
\begin{lemma}\label{lcS1}
The pair $(S_1, C)$ has at worst log canonical singularities and $K_{S_1}+C$ is  
$\pi_{\tilde{w}}$-ample.
\end{lemma}
\begin{proof}
This statement is proof on  \cite[Cor. 3.10]{ishiimodels}) and
\cite[2.3.2]{wahl1988deformations}.  
\end{proof}
The singularities of $S_T$ are determined by the weighted projective space $\tilde{E}_{\mathbf{w}}$.
\begin{lemma}
\label{sgst} Let $S_T  \subset \mathbb{P}(w_0,w_1,w_2,w_3)$ be a generic hypersurface
(see Remark \ref{gd}), then its singular locus is obtained by intersecting $S_T$ with   the edges  $P_iP_j$ and with the vertex $P_i$. In the first case $S_T$ has 
$ \lfloor gcd(w_j,w_i) \deg_{\textbf{w}}(f)/w_jw_i \rfloor$ cyclic quotient singularities of type 
$$
\frac{1}{gcd(w_i,w_j)}(1,c_{ij}) 
$$ 
where $c_{ij}$ is the solution of $w_l-w_kc_{ij} \equiv 0 \mod (gcd(w_i,w_j)), \; \; k < l$.  The other singularities of $S_T$ are induced by the vertex $P_i$ and they depend on the intersection between $S_T$ and the vertex.
\end{lemma}
\begin{proof}
By hypothesis $E_{\mathbf{w}}$ is well formed, then its singular locus has codimension 2 and it is supported on the edges $P_iP_j$.    By quasismoothness of $X$, the proof reduces to describe the intersection of  $S_T$ with  $\operatorname{Sing}(E_{\mathbf{w}})$. This is a standard calculation on weighted projective spaces see for example  \cite[Lemma 4.1]{yonemura} or \cite[I.7.1]{wps}; a useful Lemma for counting the points on  $ S_T \cap P_iP_j$ is \cite[I.6.4]{wps}. 
\end{proof}

\begin{rmk}\label{gd} 
We suppose the associated equation to the singularity is  \emph{non degenerated with respect to the Newton polytope}. Therefore, the log canonical, canonical, and minimal model of our singularities can be constructed by a subdivision of the dual fan of their respective Newton polytope (see  \cite{ishiimodels}).  We \emph{will suppose the surface
$S_T$  is quasismooth}. Then, its singularities  are induced only by the cyclic quotient singularities of $\mathbb{P}(w_0,w_1,w_2,w_3)$ (see \cite[I.5]{wps}).
We will also suppose that $S_T$ is well formed, so we can  apply the adjunction formula to find its canonical class (for details see \cite[I.3.10]{wps}).  Finally, we \emph{suppose that $S_0$ and $X$ are generic enough with respect to their weights}. So the singularities of $S_T$ depend only on the weights $\mathbf{w}$.  In particular, for each $i$, the equation contains a term of the form $x_i^n$ or $x_i^mx_k$ with a nonzero coefficient. 
\end{rmk}
\begin{ex} (see \cite[pg. 72]{watanabe1998three}) Let $E_{12}$ be a minimal elliptic surface singularity with associated equation $(x^2+y^3+z^7=0)$. Consider the smoothing 
$$
x^2+y^3+z^7 +a_1t^{42}+a_2xyzt=0  
\; \text{ with }\; 
[a_1:a_2] \in \mathbb{P}^1.
$$ 
If $a_1 \neq 0$, then $S_T \subset \mathbb{P}(21,14,6,1)$  is a $K3$ surface of weighted degree $42$ with $A_1+A_2+A_6$ singularities. On other hand if $a_1=0$, then the exceptional divisor is a rational surface on $\mathbb{P}(21,14,6,1)$ with a $T_{2,3,7}$ singularity. We highlight that this weighted blow up is the simultaneous canonical modification of these  singularities (see \cite[Ex 3.4]{ishiicw})
\end{ex}
\begin{rmk}[Motivation, see Hassett \cite{hassett}]\label{hst} Let $C_0$ be the germ of an isolated reduced plane curve singularity with the same topological type as  
$(f(x,y)=x^p+y^q=0)$. Let  $S$ be the  smoothing  defined by 
$$
(f(x,y) +t^{lcm(p,q)}=0) \to \operatorname{Spec} (\mathbb{C}[[t]] ),
$$
and let $\pi_{\mathbf{w}}$ be the induced weighted blow up with respect the weights 
$$
(w_0,w_1,w_3)= \left( \frac{q}{gcd(p,q)} , \frac{p}{gcd(p,q)} ,1 \right).
$$
a similar setting to ours is  (the symbol $\mathbb{\hat{C}}^n$ is to indicate we are working locally). 
\begin{displaymath}
\xymatrix{
C_1 \cap \mathbb{P}(w_0,w_1) \subset  C_1 +\mathbb{P}(w_0,w_1) \subset Bl_{\tilde{w}}\hat{\mathbb{C}}^2
 \ar@{->}[r] \ar[d]_{\pi_{\tilde{w}}} 
& \tilde{S}  + \mathbb{P}(w_0,w_1,1) \subset  
Bl_{\mathbf{w}}\hat{\mathbb{C}}^3 \ar[d]^{\pi_{\mathbf{w}}} 
\\
C_0 \subset \hat{\mathbb{C}}^3
\ar@{->}[r]  & 
S \subset \hat{\mathbb{C}}^3
}
\end{displaymath}
The central fiber $\tilde{S}|_0$  decomposes as the union of  two curves $C_1$ and $C_T$. The curve $C_1$ is the normalization of $C_0$ and  
$C_T = \tilde{S} \cap \mathbb{P}(w_0,w_1,1)$ is the exceptional curve contained on the proper transform of $S$.   The key point is that all the fibers of $\tilde{S}$ are Deligne-Mumford stable \cite[Thm 6.2]{hassett}
Therefore, this weighted blow up give us the local stable reduction for those plane curve singularities. 
\end{rmk}

\subsection{Smoothing associated to the monodromy}\label{sthmono}
Let  $X_0 \to  \operatorname{Spec} \left( \mathbb{C}[[\tau]] \right)$ be a generic one dimensional smoothing of $S_0$,  then its analytical form is $f(x,y,z)=t$ where $f(x,y,z)$ is an equation defining $S_0$.  We construct another smoothing $X \to \Delta$ of $S_0$ by taking a base change $t^m \to t$.  Our base change is singled out by the monodromy theorem which implies  a smoothing $X \to \Delta$ with a semistable family $Y$ dominating it,  
$Y \to X \to \Delta$, has unipotent local monodromy.  
\begin{defn}
Let $X_0 \to \operatorname{Spec}(\mathbb{C}[[\tau]])$ as before.  The unipotent base change 
is the one given by $\tau \to t:=\tau^{m}$ where $m$ is the minimum integer such that:
$$
\Delta:=\operatorname{Spec} \left( \mathbb{C}[[t]] \right) \to \operatorname{Spec}(\mathbb{C}[[\tau]])
$$ 
 induces a family $ X:=X_0 \times_{\Delta_0} \Delta \to \Delta$ with  an unique quasihomogeneous singularity and  unipotent monodromy.
\end{defn}
Varchenko proved  that if $S_0:=(f(x,y,z)=0))$    is non degenerate with respect to its Newton polyope, then the characteristic polynomial of the monodromy of $f(x,y,z)$ at the origin depends only on its associated weights.

\begin{lemma}\label{smtg}
Let $S_0$ be a quasihomogeneous non degenerated singularity, then its   unipotent base change is induced by its weighted degree $\deg_{\tilde{w}}(f)$.
\end{lemma}
\begin{proof}
Let $\xi_i$ be the eigenvalues of the classical monodromy operator associated to a smoothing of $S_0$; this monodromy is unipotent if and only if $\xi_i^m=1$ for all $i$.  We must prove that if 
$m=  \deg_{\tilde{w}}(f)$, then $\xi_i^m=1$, and that  $\deg_{\tilde{w}}(f)$ is the smallest possible integer with that property.  Since $\xi_i$ are also the roots of the characteristic polynomial $\theta_{S_0}(t)$ of the monodromy, our statement follows directly from an expression, due to Ebeling  \cite[Thm 1]{ebeling2002poincare}:
$$
\theta_{S_0}(t) = \frac{(1-t^d)^{2g-2+r} \prod_{w_j |d} (1-t^{d/w_j}) }{ (1-t) \prod_{\alpha_i | d}(1-t^{d/\alpha_i}) } 
$$ 
where $d =\deg_{\tilde{w}}(f)$,  $g$ is the genus of $C$, $\alpha_i$ and $r$ are as on Remark \ref{dol}.  Another explicit expression of $\theta_{S_0}(t)$ in terms of the weights is given on  \cite[Prop. 2.2]{xu1989classification}
\end{proof}
%The exceptional surface $S_T$ is not contained in the singular locus of  $\tilde{X}$ (see \cite[pg. 5]{wahl2011log}) 
\subsection{Smoothing associated to the dual sets of $\mathcal{B}_{S_0}$} Let 
$X \subset \mathbb{C}^4$ be a smoothing of $S_0$ with an isolated quasihomogeneous singularity, by taking the weighted blow up of  $X$ we obtain an one dimensional family $\tilde{X}$ of surfaces  degenerating to the union of $S_1$ and $S_T$.  We
 are interested  whenever the surface $S_T$ has $\mathcal{\hat{B}}_{S_0}$ singularities.
\begin{defn}\label{sB}
Let $X \subset \mathbb{C}^4$ be  a  smoothing of $S_0$ with a good $\mathbb{C}^*$-action and let  $\hat{\mathcal{B}}_{S_0}$ be one of the dual sets of $ \mathcal{B}_{S_0}$.  We say that  $X$  is a smoothing associated to $\hat{\mathcal{B}}_{S_0}$ if the following conditions holds:
\begin{enumerate}
\item By taking the associated weighted blow up of $X$, we obtain a threefold  $\tilde{X}$ which central fiber decomposes as $\tilde{X}|_0= S_1 +S_T$ where $S_1$ is the proper transform of $S_0$ and $S_T$ is a well-formed, quasismooth surface in a weighed projective space.
\item The singular locus of $S_1$ is $\mathcal{B}_{S_0}$
\item The  singular locus of $S_T$ is  $\hat{\mathcal{B}}_{S_0}$ and maybe some additional DuVal singularities.
\end{enumerate}
We denote this smoothing as $ X(\hat{\mathcal{B}}_{S_0})$
\end{defn}
\begin{ex}
The set $(\hat{\mathcal{B}}_{S_0})_2$ described on Example \ref{notunq} does not have an associated smoothing $X \subset \mathbb{C}^4$ as defined above.  Indeed, the possible weights must be of the form 
$\mathbf{w}=(33,22,6,w_3)$.  A direct calculation shows that the condition 
$$
(\hat{\mathcal{B}}_{S_0})_2 \subset Sing ( \mathbb{P}(33,22,6,w_3) ) \cap S_T
$$ 
implies that $w_3 \equiv 49 \mod ( 66)$. By taking a linear combination of monomials 
quasihomogeneous with respect those weights, it is not possible to construct a smoothing that satisfies the conditions of Definition \ref{sB}. 
\end{ex}

\section{Type I Quasihomogeneous Singularities}\label{I}

For this family of surface singularities the setting of Section  \ref{wbup} allows us to create a family of surfaces $\tilde{X}$ degenerating to a central fiber  with singularities $\mathcal{B}_{S_0}$ and $\mathcal{\hat{B}}_{S_0}$.
\begin{thm}\label{thm1}
Let  $X$ be the unipotent smoothing of a  quasihomogeneous  singularity $S_0$ of type I; let 
$\tilde{X}$ be the proper transform of $X$ under the weighted blow up 
$\pi_{\textbf{w}}$. Then it holds that: 
\begin{enumerate}
\item $\tilde{X}$ has at worst  terminal cyclic quotient singularities 
\item  The central fiber  $\tilde{X}_0$  has orbifold double normal crossing singularities, and it decomposes in two surfaces $S_1+S_T$.
\item Let $
\mathcal{B}_{S_0} =\left\{ 
\frac{1}{r_i}(1,b_i) \right\}
$ be  the singular locus of $S_1$.  Then, the unipotent smoothing is the one associated to the dual set  $\mathcal{\hat{B}}_{S_0}
:=\left\{ 
\frac{1}{r_i}(1,-b_i)
\right\}
$
\end{enumerate}
\end{thm}
\begin{rmk}
An orbifold double normal crossing singularity is locally of the form
$$
(xy=0) \subset \frac{1}{r_i}(1,-1,c_i), \; \; \; \; (r_i,c_i)=1
$$ 
\end{rmk}
\begin{rmk}
In general for a surface singularity $S_0$ the set $\mathcal{B}_{S_0}$ have several  dual sets $\mathcal{\hat{B}}_{S_0}$ with different associated smoothings (see Section \ref{hm}).  Nevertheless, for unimodal singularities a sense of uniqueness is accomplished (see Theorem \ref{thm2}).
\end{rmk}
\begin{proof}
This result  follows from Ishii's characterization of canonical modifications \cite{ishiicw}. Next, we describe her approach.  The notation is the standard one in toric geometry. Let $X$ be an isolated hypersurface singularity defined by a non degenerated quasihomogeneous polynomial   $g:=\sum_{\textbf{a} \in M} c_{\textbf{a}}x^{\textbf{a}}$. Oka proved that we can obtain a resolution of $X$ by making a subdivision $\Sigma_0$ of the  dual fan of the Newton polytope $\Gamma(g) \subset N_{\mathbb{R}}$   (see \cite{oka}).  
 This subdivision $\Sigma_0$ is induced by primitive vectors $\mathbf{p}_i$ on 
$ N_{\mathbb{R}}$. From the fan associated to the  subdivision $\Sigma_0$, we obtain a toric variety  $T(\Sigma_0)$ such that the proper transform  $X(\Sigma_0)$ of $X$ is smooth, intersects transversely each orbit and 
$$ 
K_{ X(\Sigma_0)}= \varphi^*(K_X) + \sum_{\mathbf{p}_i \in \Sigma_0(1)-\sigma(1)}
 a(\mathbf{p}_i, X)E_{\mathbf{p}_i}|_{X(\Sigma_0)}
$$
where $E_{\mathbf{p}_i}$ are the exceptional divisor associated to the primitive vector $\mathbf{p}_i$, the vectors $\mathbf{p}_i  \in \Sigma_0(1)-\sigma(1)$ are the new rays added to the fan,  and 
$$
a(\mathbf{p}_i, X) =
 \sum_k (\mathbf{p}_i)_k - 
\min  \left\{  \mathbf{p}_i(\mathbf{a})  \;  \Big|  \; \mathbf{a} \in M, 
g:=\sum_{\textbf{a} \in M} c_{\textbf{a}}x^{\textbf{a}},  c_{\textbf{a}} \neq 0  \right\}
:= \mathbf{p}_i(\mathbf{1})-\mathbf{p}_i(g)-1
$$
with $\mathbf{p}(\mathbf{a}) = \sum_k p_ka_k$. From our purposes, it is enough to consider the primitive vectors $\mathbf{p}_i$ inside the essential cone of the singularity: 
$$
C_1(g):= \{ \mathbf{s} \in N_{\mathbb{R}} \; | \;  -1 \geq a(\mathbf{s}, X) \; \; \;
\text{ and }\;\; \; s_i \geq 0 
 \}.
$$
because if $ \mathbf{p} \notin C_1(g)$, then $a(\mathbf{p}, X) \geq 0$.    To prove that $\pi_{\mathbf{w}}$ is the canonical modification of $X$, we must show that
$\mathbf{w} \in C_1(g)$ and that  for any $ \mathbf{s}  \in C_1(g) $, $ \mathbf{s} \neq \mathbf{w} $ the discrepancy  associated to $E_{\mathbf{s}}$ is non negative.  
This translates into a combinatorial condition between
 $ \mathbf{w}$ and $\mathbf{s}$  (see \cite[Thm 2.8]{ishiicw});  the  weighted blow up 
$\pi_{\mathbf{w}}$ is the canonical modification of $X$ if and only if
$\mathbf{w} \in C_1(g)$  and $\mathbf{w}$ is $g$-minimal in $C_1(g) \cap N \setminus \{ 0\}$.   The $g$-minimality of  $\mathbf{w}$ means that for all $\mathbf{s} \in C_1(g)$ one of the two following inequalities holds for all  $i \in \{ 1 \ldots n \}$:
\begin{align}
\mathbf{w} \leq_g \mathbf{s}
:=
&
\frac{w_i}{\mathbf{w}(g)-\mathbf{w}(\mathbf{1})+1}
\leq \frac{s_i}{\mathbf{s}(g)-\mathbf{s}(\mathbf{1})+1}
\\
\mathbf{w} \preceq_g \mathbf{s}
:=
&
\frac{w_i}{\mathbf{w}(g)} 
\leq \frac{s_i}{\mathbf{s}(g)}
\end{align}
and $\mathbf{s}$ belongs to the interior of a $(n+1)$-dimensional cone of $\sigma(\mathbf{w})$. In our case, the equation is given by 
$$ 
g(x,y,z,t):=  x^{p_0}+y^{p_1}+z^{p_2}  +t^{\deg_{w}(S_0)}
+ \sum c_{i_0,i_1,i_2,i_3}x^{i_0}y^{i_1}z^{i_2}t^{i_3} 
$$ 
with $ p_0 \geq p_1 \geq p_2 \geq 2$,   $\mathbf{w} =(w_0,w_1,w_3,1)$ where $w_i$ is as on Example \ref{wgt} and  $i_0w_0+i_1w_1+i_2w_2+i_3 = \deg_{w}(S_0) $. The claim that $\mathbf{w} \in C_1(g)$ follows from an inductive argument and by ruling out very low values of the exponents,  such as $p_0=p_1=p_2=2$, which define log canonical surface singularities. By the definition of unipotent smoothing
$$
\frac{\mathbf{w}}{\mathbf{w}(g)}
=
\left(
\frac{1}{p_0} ,  \frac{1}{p_1}, \frac{1}{p_2} , \frac{(p_1p_2,p_0p_2,p_0p_1)}{p_0p_1p_2}
\right)
$$
Therefore, $g$-minimality follows at once from the definition of weighted degree: 
$$
\mathbf{s}(g)=\min \left\{ s_0p_0, s_1p_1, s_2p_2, s_3
\frac{p_0p_1p_2}{(p_1p_2,p_0p_2,p_0p_1)} \right\}
$$
Morever,   for $\mathbf{w}$ and any $\mathbf{s} \in C_1(g)$, it holds
$ \mathbf{w} \preceq \mathbf{s} $. This implies the singularities at $\tilde{X}$ are terminal.  Indeed, let $\Sigma$ be a non singular subdivision of  the fan $\Delta(\mathbf{w})$ associated to 
$Bl_{\textbf{w}} \mathbb{C}^n$, let $\phi$ be the proper birational morphism associated to this subdivision.  
$$
  \xymatrix{X(\Sigma) \ar[r]^{\phi} & \tilde{X} \ar[r]^{\pi_{\mathbf{w}}} &   X 
  }
$$
 Let $U_i \subset Bl_{\textbf{w}}  \mathbb{C}^n$ be an open set associated to the subfan 
$\Delta_i \subset \Delta(\mathbf{w})$, and let 
$\Sigma_i \subset \Sigma$ be the preimage of $\Delta_i$  on the non singular subdivision.
We can take the restriction $X_i:=\tilde{X} \cap U_i$ and its proper transform under the resolution $X(\Sigma_i) := \phi_*^{-1}(X_i)$ to obtain (see \cite[Prop. 2.6]{ishiicw})
$$
K_{X(\Sigma_i)} = \phi^*(K_{U_i}+X_i) \big|_{X(\Sigma_i)}
+\sum m_{\mathbf{s}} \left( E_{\mathbf{s}} \cap X(\Sigma_i) \right)_{\scriptstyle{red}}
$$
where $\mathbf{s} \neq \mathbf{w}$ and 
$$
m_{\mathbf{s}} = \frac{s_i}{w_i} \left( \mathbf{w}(f)-\mathbf{w}(\mathbf{1})+1 \right) -
\left( \mathbf{s}(f)-\mathbf{s}(\mathbf{1})+1 \right)
$$
In the proof of Theorem 2.8 at \cite{ishiicw} we found that:
\begin{enumerate}
\item If $ \mathbf{s} \in C_1(g)$ and $ \mathbf{w} \preceq_g \mathbf{s} $, then $E_{\mathbf{s}} \cap X(\Sigma_0) =\emptyset$ 
\item If $\mathbf{s} \notin C_1(g)$ then  $m_\mathbf{s} >0$.
\end{enumerate}
In our case  $ \mathbf{w} \preceq_g \mathbf{s} $ for all $\mathbf{s} \in C_1(g)$, then for any $\mathbf{s}$ we have either $m_s > 0$ or $E_{\mathbf{s}} \cap X(\Sigma_0) =\emptyset$.   This implies the singularities on $\tilde{X}$ are terminal, because resolving a canonical singularity will induce an exceptional divisor $E_s$ with $m_s=0$ and 
$E_s \cap X(\Sigma_0) \neq \emptyset$.  Terminal singularities are of codimension three, so they are isolated on $\tilde{X}$. By construction  the singularities on $\tilde{X}$ are cyclic quotient ones and caused solely by the $\mathbb{C}^*$-action. 
Finally, $\tilde{X}_0$ is reduced and it decomposes into two surfaces $S_1$ and $S_T$. The singularities of $S_1$ are the ones in  $\mathcal{B}_{S_0}$ (see  Example \ref{bf} for an explicit expression).  The singularities of $S_T$ are  calculated in Lemma \ref{sgst}. Our statement follows from those expressions.
%ralbovsky
\end{proof}
\begin{ex} The $W_{15}$ singularity (also known as $ A(1,-2,-2,-3,-3)$) is a Fuchsian bimodal singularity and its normal form is  $x^2+y^4+z^6$. The unipotent base change $t^{12} \to t$ induces the $8$th case on Yonemura's classification \cite{yonemura}.   The set of singularities are: 
\begin{align*}
\mathcal{B}_{W_{15}} = 
\left\{
2 \times \frac{1}{2}(1,1),  2 \times \frac{1}{3}(1,1)
\right\}
& & 
\mathcal{\hat{B}}_{W_{15}} = 
\left\{
2 \times \frac{1}{2}(1,1),  2 \times \frac{1}{3}(1,2)
\right\}
\end{align*}
This is the only dual set realized by a smoothing. In this case, the exceptional surface is a $K3$ surface $S_T:=S_{12} \subset \mathbb{P}(6,3,2,1)$
\end{ex}

%%%%%%%%%%%%%%%%%%%%%%%%%%%%%%%%%%
\section{Unimodal Singularities} \label{unimod}

We focus on  unimodal non log canonical singularities.  It is well known that there are 14 of those singularities, and that they are all quasihomogeneous.  For more details about them, see Arnold \cite[pg 247]{arnold}, Laufer \cite{lme}, and Dolgachev \cite{dolgachev1974quotient}. 

\begin{thm} \label{thm2}
Let $S_0$ be an unimodal surface singularity, then it holds:
\begin{enumerate}
\item From all the canonical dual sets associated to $\mathcal{B}_{S_0}$, there is only one dual set $\mathcal{\hat{B}}_{S_0}$  with an associated smoothing 
$X(\mathcal{\hat{B}}_{S_0}) \to \Delta$ as on  Definition \ref{sB}.
\item The smoothing $X(\mathcal{\hat{B}}_{S_0}) $ coincides with the one induced by the unipotent base change.  Moreover, this threefold has an unique strictly log canonical  singularity.
\item The threefold $\tilde{X}$ has isolated terminal cyclic quotient singularities. 
\item The central fiber $\tilde{X}_0$  has orbifold double normal crossing singularities and it decomposes  in two surfaces $S_1$ and $S_T$ intersecting along a rational curve 
$C=S_1 \cap S_T$ which satisfies
$$
\left( S_1 \big|_{S_T} \right)^2  =  \frac{\deg_{\mathbf{w}}(f)}{w_0w_1w_2} 
$$
\item  $S_1$ is the proper transform of $S_0$ and it supports the singularities in $\mathcal{B}_{S_0}$.
\item $S_T$ is a K3 surface and it supports the singularities in $\mathcal{\hat{B}}_{S_0}$. 
\item The line bundles $K_{\tilde{X}}|_{S_1}$ and $K_{\tilde{X}}|_{S_T}$ are ample. 
\end{enumerate}
\end{thm}
\begin{rmk}\label{sk3}
Yonemura \cite{yonemura} classified all the hypersurface threefold singularities which exceptional surface is a  normal $K3$ surface with canonical singularities.  There are 95 of those families and they are in bijection with the list of 95 normal K3 surfaces that appear as a hypersurface in a weighted projective space.
\end{rmk}
\begin{rmk}\label{mono}
The relationship between monodromy and smoothings of surfaces with semi log canonical (slc) singularities is not straightforward.  Indeed,  let $S_1 \cup S_T$ be a surface with at worst slc singularities. Then, $S_T$ and $S_1$ can have cyclic quotient singularities away from their intersection.  Any cyclic quotient singularity is log terminal and they can induce arbitrary large monodromy to a generic smoothing of these surfaces.  On other hand, the hypothesis that there is a semistable family dominating the smoothing is used in the proof of important theorems related to slc surfaces (For example \cite[Thm 5.1]{KSB88}). The  monodromy theorem implies that those families have unipotent smoothings.
\end{rmk}
\begin{proof}
Let $S_0$ be an unimodal singularity with associated weights $(w_0,w_1,w_2)$, the first statement claims that there is only one $1 \leq w_3 \leq \deg_{\tilde{w}}(S_0)$ and one dual set $\mathcal{\hat{B}}_{S_0}$ for which the singularities induced on 
$S_{T} \subset \mathbb{P}(w_0,w_1,w_2,w_3)$ are the ones in $\mathcal{\hat{B}}_{S_0}$.  The second and third statement means that such unique set of weights is $\textbf{w}=(w_0,w_1,w_2,1)$, and that  the cyclic quotient singularities on our $\mathcal{\hat{B}}_{S_0}$ are of type $\frac{1}{\alpha_i}(1,\alpha_i-1)$ (see Example \ref{dfc}).  These statements follows from an individual study of each singularity and their associated weighted projective spaces; this is described on the rest of the section. In fact, we wrote a small computer program in Sage \cite{sage} to compute the singularities of  $S_T$, to compare them with the different $\mathcal{\hat{B}}_{S_0}$, and to test the quasismoothness and well formedness of our possible exceptional surfaces.  The quasismothness of $X$ implies  that $\tilde{X}$ has only cyclic quotient singularities
(see \cite[Lemma 8]{fujiki}) which are terminal by the nature of $\mathcal{B}_{S_0}$ and 
$\mathcal{\hat{B}}_{S_0}$. Those terminal singularities are induced by  
$ \tilde{X} \cap Sing  \left( \mathbb{P}(w_0,w_1,w_2,1) \right)$ and they are isolated.  We can see that $S_T$ is a K3 surface by the adjunction formula. We remark that the smoothing $X \to \Delta$ and its partial resolution $\tilde{X} \to X \to \Delta$ has been studied by Yonemura \cite{yonemura}, Ishii, and Tomari in the context of simple K3 surface singularities.   Several of our claims follow by their work. In particular, they show that the weighted blow up is the terminal modification of $X$, and that in fact it is unique if $X$ is defined by a generic polynomial (see \cite[Thm. 3.1]{yonemura}).  The relationship between unimodal singularities and Yonemura's classification is described, in another context, by Prokhorov
\cite{prokhorov}.

We find the value of $\left( S_1 \big|_{S_T} \right)^2$ by using Lemmas \ref{tC2} and \ref{C2}.  From adjunction formula, the fact that $S_T$ is a normal K3 surface, and that the fibers are numerically equivalent we have:
\begin{align*}
K_{ \tilde{X}}\big| _{S_1} = K_{S_1} + S_T\big|_{S_1}
& &
K_{ \tilde{X}}\big| _{S_T} =  K_{S_T}+ S_1 \big|_{S_T} =  S_1 \big|_{S_T} 
\end{align*}
The surface $S_T$ holds $Pic(S_T) \cong \mathbb{Z}$. Therefore,
$
 K_{\tilde{X}}\big| _{S_T}$ is ample because  its degree is positive. The ampleness of $K_{ \tilde{X}}\big| _{S_1} $ follows from Lemma \ref{lcS1}
\end{proof}

\begin{rmk} \label{badu12}
The previous statements are not longer true for all higher modal singularities.  See 
Section \ref{hm} for details.  Moreover, the weighted blow up of an arbitrary quasihomogeneous smoothing does not necessarily yields a threefold with canonical or terminal singularities (see Remark \ref{u12}).  
\end{rmk}
Next, we give an explicitly description of the central fiber $\tilde{X}|_0$ for each unimodal singularity.  We describe the most details for the $E_{12}$ singularity.  The other cases are similar.
\subsection{The $E_{12}$ singularity}\label{e12}
 It is also known as $D_{2,3,7}$ or $Cu(-1)$. Its normal form is $x^2+y^3+z^7$, and its unimodal base change $t^{42} \to t$ induces the $20$th case of Yonemura's classification.  The surface $S_1$ supports the cyclic quotient singularities: 
$$
\mathcal{B}_{E_{12}}
=
\left\{
\frac{1}{2}(1,1), \frac{1}{3}(1,1), \frac{1}{7}(1,1)
\right\}
$$
The possible dual sets are
\begin{align*}
\mathcal{\hat{B}}^1_{E_{12}} =
\left\{
\frac{1}{2}(1,1), \frac{1}{3}(1,2), \frac{1}{7}(1,5)
\right\}
& &
\mathcal{\hat{B}}^2_{E_{12}} = 
\left\{
\frac{1}{2}(1,1), \frac{1}{3}(1,1), \frac{1}{7}(1,5)
\right\}
\\
\mathcal{\hat{B}}^3_{E_{12}} =
\left\{
\frac{1}{2}(1,1), \frac{1}{3}(1,1), \frac{1}{7}(1,6)
\right\}
& &
\mathcal{\hat{B}}^4_{E_{12}} = 
\left\{
\frac{1}{2}(1,1), \frac{1}{3}(1,2), \frac{1}{7}(1,6)
\right\}
\end{align*}
The singularities on $\mathcal{\hat{B}}^4_{E_{12}}$  are the only ones that can be realized by a weighed blow up such that  the induced surface $S_{42} \subset \mathbb{P}(21,14,6,w_3)$
is quasismooth and well-formed as on Definition \ref{sB}.  In that case $w_3=1$, and $S_T$ is a K3 surface.  The surfaces in the central fiber intersect along a rational curve $C= S_T  \cap S_1$.  (see Figure \ref{fig1})
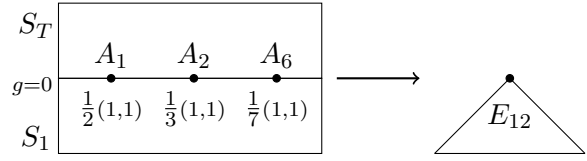
\begin{figure}[h!]
\begin{center}
\begin{tikzpicture}
\draw (4,0) -- (6,0); %hypotenuse
\draw (5,1) -- (4,0);
\draw (5,1) -- (6,0); %opposite
\filldraw (5,1) circle [radius=0.05];
%%%
\node at (3,1.3) { };
\draw[ thick] [->] (2.7,1) -- (3.8,1);
%%%
\draw (-1,0) rectangle (2.5,1);
\draw (-1,1) rectangle (2.5,2);
% points
\filldraw (-0.3,1) circle [radius=0.05];
\node at (0.3,1.3) { };
\filldraw (0.8,1) circle [radius=0.05];
\node at (0.8,1.3) { };
\filldraw (1.9,1) circle [radius=0.05];
\node at (1.9,1.3) { };
% points
\node at (5,1.3) { };
\node at (-0.8,0.3) { };
\node at (-0.8,1.6) { };
\node at (5,0.3) { };
% labels
\node at (-1.35,0.9) { $\scriptstyle{g=0}$};
\node at (-0.3,1.3) { $A_1$};
\node at (0.8,1.3) { $A_2$};
\node at (1.9,1.3) { $A_6$};
\node at (-0.3,0.6) { $\frac{1}{2}\scriptstyle{(1,1)}$};
\node at (0.8,0.6) {$\frac{1}{3}\scriptstyle{(1,1)}$ };
\node at (1.9,0.6) { $\frac{1}{7}\scriptstyle{(1,1)}$};
%%
%\node at (5,1.3) { $E_{12}$};
\node at (5,0.5) { $E_{12}$};
\node at (-1.3,1.7) { $S_{T}$};
\node at (-1.3,0.2) { $S_{1}$};
\end{tikzpicture}
\caption{Analysis of the $E_{12}$ singularity}\label{fig1}
\end{center}
\end{figure}

\subsection{The $E_{13}$ singularity} It is also known as $D_{2,4,5}$ or $Ta(-2,-3)$. Its 
normal form is  $x^2+y^3+yz^5$, and its unipotent smoothing $t^{30} \to t$ induces the $50$th case on Yonemura's classification. The set of singularities are
\begin{align*}     
\mathcal{B}_{E_{13}} =  \left\{  \frac{1}{2}(1,1), \frac{1}{4}(1,1), \frac{1}{5}(1,1) \right\}
& &
\mathcal{\hat{B}}_{E_{13}} =  \left\{ 
 \frac{1}{2}(1,1), \frac{1}{4}(1,3), \frac{1}{5}(1,4) 
\right\}
\end{align*}
The associated K3 surface is $S_{30} \subset \mathbb{P}(15,10,4,1)$.

\subsection{The $E_{14}$ singularity} It is also known as $D_{3,3,4}$ or $Tr(-2,-2,-3)$. Its is normal form is $x^2+y^3+z^8$, and its unipotent base change $t^{24} \to t$ induces  the $13$th case on Yonemura's classification. The set of singularities are:
\begin{align*}
\mathcal{B}_{E_{14}} = \left\{
2 \times \frac{1}{3}(1,1), \frac{1}{4}(1,1)
\right\}
 & & 
\mathcal{\hat{B}}_{E_{14}} = \left\{
2 \times \frac{1}{3}(1,2), \frac{1}{4}(1,3)
\right\}
\end{align*}
The associated K3 surface is $S_{24} \subset \mathbb{P}(12,8,3,1)$.
\begin{rmk} The set of weights 
$ \textbf{v} =  (12, 8, 3, 13) $ seems to induce the singularities of the set $\mathcal{\hat{B}}_{E_{14}} $. However,  in this case the surface 
$S_{24} := (x^2+y^3+z^8+yzt=0) \subset \mathbb{P}(12,8,3,13)$  is not quasismooth.   
In fact, the surface has a $A_1$ singularity supported on the vertex $P_3$ which itself supports the singularity  $\frac{1}{13}(1,5,10)$. 
\end{rmk}

\subsection{The $U_{12}$ singularity} 
It is also known as $D_{4,4,4}$ or $Tr(-3,-3)$. Its normal form is $x^3+y^3+z^4$, and its unipotent base change  $t^{12} \to t$ induces the 4th case on the Yonemura classification.  The set of singularities are: 
\begin{align*}
\mathcal{B}_{U_{12}}
=
\left\{
3 \times \frac{1}{4}(1,1)
\right\}
 & & 
\mathcal{\hat{B}}_{U_{12}}
=
\left\{
3 \times \frac{1}{4}(1,3)
\right\}
\end{align*}
The associated K3 surface is $S_{12} \subset \mathbb{P}(4,4,3,1)$.
\begin{rmk}\label{u12} (see Remark \ref{badu12})  The smoothing induced by the weights $\textbf{v} = (4,4,3,3)$ induces the exceptional surface
$$
S_{12}: =(x^3+y^3+z^4+t^4=0) \subset \mathbb{P}(4,4,3,3)
$$ 
supporting the singularities
$$
\left\{
 3 \times \frac{1}{4}(1,1), 4 \times \frac{1}{3}(1,1)
\right\}
$$
The threefold  supports the singularities $\frac{1}{4}(1,1,1)$. Next, we apply the Reid-Tai criterion to the associated group generator 
$
\epsilon_{4} (x,y,z) \to (\epsilon_4x, \epsilon_4 y, \epsilon_4 z) 
$. 
The age of $\epsilon_4$ (see \cite[105]{smmp}) is $3/4 <1$ which implies the singularity is not canonical.
\end{rmk}
\begin{rmk} \label{ru12}
The set of weights 
$\textbf{v} = (4, 4, 3, 9)$ induces the quasismooth surface
$S_T:= (x^3+y^3+z^4+zt =0) \subset \mathbb{P}(4, 4, 3, 9)$. The induced singularities on $S_T$  
 are
$$
\left\{
3 \times \frac{1}{4}(1,3), \frac{1}{3}(1,1), \frac{1}{9}(1,1)
\right\}
$$
the problem here is that we have additional non DuVal singularities on our exceptional tail.
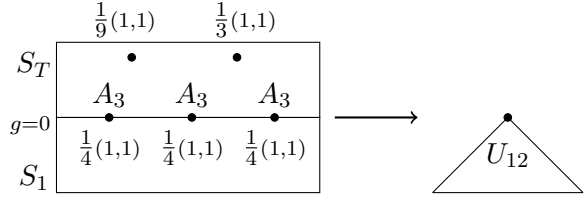
\begin{figure}[h!]
\begin{center}
\begin{tikzpicture}
\draw (4,0) -- (6,0); %hypotenuse
\draw (5,1) -- (4,0);
\draw (5,1) -- (6,0); %opposite
\filldraw (5,1) circle [radius=0.05];
%%%
\node at (3,1.3) { };
\draw[ thick] [->] (2.7,1) -- (3.8,1);
%%%
\draw (-1,0) rectangle (2.5,1);
\draw (-1,1) rectangle (2.5,2);
% points
\filldraw (-0.3,1) circle [radius=0.05];
\node at (0.3,1.3) { };
\filldraw (0.8,1) circle [radius=0.05];
\node at (0.8,1.3) { };
\filldraw (1.9,1) circle [radius=0.05];
\node at (1.9,1.3) { };
\filldraw (1.4,1.8) circle [radius=0.05];
\filldraw (0,1.8) circle [radius=0.05];
% points
\node at (5,1.3) { };
\node at (-0.8,0.3) { };
\node at (-0.8,1.6) { };
\node at (5,0.3) { };
% labels
\node at (-1.35,0.9) { $\scriptstyle{g=0}$};
\node at (-0.1,2.3) { $\frac{1}{9}\scriptstyle{(1,1)}$};
\node at (1.5,2.3) { $\frac{1}{3}\scriptstyle{(1,1)}$};
\node at (-0.3,1.3) { $A_3$};
\node at (0.8,1.3) { $A_3$};
\node at (1.9,1.3) { $A_3$};
\node at (-0.3,0.6) { $\frac{1}{4}\scriptstyle{(1,1)}$};
\node at (0.8,0.6) {$\frac{1}{4}\scriptstyle{(1,1)}$ };
\node at (1.9,0.6) { $\frac{1}{4}\scriptstyle{(1,1)}$};
\node at (5,0.5) { $U_{12}$};
\node at (-1.3,1.7) { $S_{T}$};
\node at (-1.3,0.2) { $S_{1}$};
\end{tikzpicture}
\caption{Remark \ref{ru12} about the $U_{12}$ singularity}\label{figu12}
\end{center}
\end{figure}

\end{rmk}

\subsection{The $W_{12}$ singularity} It is also known as $D_{2,5,5}$ or $Ta(-3,-3)$. Its is normal form is $x^2+y^4+z^5$, its unipotent base change $t^{20} \to t$  induces  the $9$th case on Yonemura's classification. The set of singularities are:
\begin{align*}
\mathcal{B}_{W_{12}} = \left\{
 \frac{1}{2}(1,1), 2 \times \frac{1}{5}(1,1)
\right\}
& &
\mathcal{\hat{B}}_{W_{12}} = \left\{
 \frac{1}{2}(1,1), 2 \times \frac{1}{5}(1,4)
\right\}
\end{align*}
where $\mathcal{\hat{B}}_{W_{12}} $ is supported on the K3 surface $S_{20} \subset \mathbb{P}(10,5,4,1)$.

\subsection{The $W_{13}$ singularity} It is also known as  $D_{3,4,4,}$ or $Tr(-2,-3,-3)$. Its normal form is  $x^2+y^4+yz^4$, and its unipotent smoothing $t^{16} \to t$ induces the $37$th case on Yonemura's classification. The set of singularities are 
\begin{align*}
\mathcal{B}_{W_{13}} = \left\{
 \frac{1}{3}(1,1), 2 \times \frac{1}{4}(1,1)
\right\}
& &
\mathcal{\hat{B}}_{W_{13}} = \left\{
 \frac{1}{3}(1,2), 2 \times \frac{1}{4}(1,3)
\right\}
\end{align*}
The associated $K3$ surface is $S_{16} \subset \mathbb{P}(8,4,3,1)$.
\begin{rmk}\label{rw13}
The smoothing obtained by taking linear combination of monomials of degree $16$ with respect the weights $(8,4,3,13)$ is associated to the set of singularities 
\begin{align*}
\left\{
\frac{1}{3}(1,2), 2  \times \frac{1}{4}(1,3),  \frac{1}{13}(1,7)
\right\}
\end{align*}
Then we discard this smoothing  because the presence of a non DuVal singularity on the surface $S_T$
\begin{figure}[h!]
\begin{center}
\begin{tikzpicture}
\draw (4,0) -- (6,0); %hypotenuse
\draw (5,1) -- (4,0);
\draw (5,1) -- (6,0); %opposite
\filldraw (5,1) circle [radius=0.05];
%%%
\node at (3,1.3) { };
\draw[ thick] [->] (2.7,1) -- (3.8,1);
%%%
\draw (-1,0) rectangle (2.5,1);
\draw (-1,1) rectangle (2.5,2);
% points
\filldraw (-0.3,1) circle [radius=0.05];
\node at (0.3,1.3) { };
\filldraw (0.8,1) circle [radius=0.05];
\node at (0.8,1.3) { };
\filldraw (1.9,1) circle [radius=0.05];
\node at (1.9,1.3) { };
\filldraw (0.8,1.8) circle [radius=0.05];
% points
\node at (5,1.3) { };
\node at (-0.8,0.3) { };
\node at (-0.8,1.6) { };
\node at (5,0.3) { };
% labels
\node at (-1.35,0.9) { $\scriptstyle{g=0}$};
\node at (0.8,2.3) { $\frac{1}{13}\scriptstyle{(1,7)}$};
\node at (-0.3,1.3) { $A_4$};
\node at (0.8,1.3) { $A_5$};
\node at (1.9,1.3) { $A_5$};
\node at (-0.3,0.6) { $\frac{1}{3}\scriptstyle{(1,1)}$};
\node at (0.8,0.6) {$\frac{1}{4}\scriptstyle{(1,1)}$ };
\node at (1.9,0.6) { $\frac{1}{4}\scriptstyle{(1,1)}$};
%%
%\node at (5,1.3) { $W_{13}$};
\node at (5,0.4) { $W_{12}$};
\node at (-1.3,1.7) { $S_{T}$};
\node at (-1.3,0.2) { $S_{1}$};
\end{tikzpicture}
\caption{Remarl \ref{rw13} about the $W_{13}$ singularity}
\end{center}
\end{figure}
\end{rmk}
%%%
\subsection{The $Q_{10}$ singularity} It  also known as $D_{2,3,9}$ or $Cu(-3)$. Its normal form is  $x^2z+y^3+z^4$, its unipotent base change $t^{24} \to t$ induces the 20th case on Yonemura classification.  The set of singularities are
\begin{align*}
\mathcal{B}_{Q_{10}} 
=
\left\{ \frac{1}{2}(1,1), \frac{1}{3}(1,1), \frac{1}{9}(1,1)\right\} 
& &
\mathcal{\hat{B}}_{Q_{10}}  
=
\left\{ \frac{1}{2}(1,1), \frac{1}{3}(1,2), \frac{1}{9}(1,8)\right\}
\end{align*}
The associated K3 surface is $S_{24} \subset \mathbb{P}(9,8,6,1)$. 
\begin{rmk}
Consider another smoothing constructed by taking linear combination of monomials of degree  $24$ with respect the weights $(9,8,6,5)$. The associated exceptional surface
has  singularities:
$$
\left\{ 
\frac{1}{2}(1,1), \frac{1}{3}(1,1), \frac{1}{9}(1,4), \frac{1}{5}(1,3)
\right\}
$$
after taking another base change $t^5 \to t$, we will induce the smoothing 
$(x^2z+y^3+z^4+xt^{15}=0)$ which a non generic smoothing on the $Q_{10}$ versal deformation space.
\end{rmk}

\subsection{The $Q_{11}$ singularity} also known as $D_{2,4,7}$ or $Ta(-2,-5)$. Its associated equation is $x^2z+y^3x+ yz^3$, its unipotent base change  $t^{18} \to t$ induces the 60th case on Yonemura classification. The sets of singularities are
\begin{align*}
\mathcal{B}_{Q_{11}} 
=
\left\{ 
\frac{1}{2}(1,1), \frac{1}{4}(1,1), \frac{1}{7}(1,1)
\right\}
& & 
\mathcal{\hat{B}}_{Q_{11}} 
=
\left\{ 
\frac{1}{2}(1,1), \frac{1}{4}(1,3), \frac{1}{7}(1,6)
\right\}
\end{align*}
the  associated K3 surface is $S_{18} \subset \mathbb{P}(7,6,4,1)$.

\subsection{The $Q_{12}$ singularity} It also known as  $D_{3,3,6}$ or $Tr(-2,-2,-5)$. Its normal form is $x^2z+y^3+z^5$, and its unipotent smoothing $t^{15} \to t$ induces the $22$th case on Yonemura's classification. The set of singularities are:
\begin{align*}
\mathcal{B}_{Q_{12}} 
=
\left\{ 
2 \times \frac{1}{3}(1,1), \frac{1}{6}(1,1)
\right\}
& & 
\mathcal{\hat{B}}_{Q_{12}} 
=
\left\{ 
2 \times \frac{1}{3}(1,2), \frac{1}{6}(1,5)
\right\}
\end{align*}
the associated K3 surface is $S_{15} \subset \mathbb{P}(6,5,3,1)$.
\begin{rmk}
A general smoothing defined by a linear combination of monomials of degree $15$ with respect the weights $(6,5,3,2)$ is given by
$$
x^2z +y^3+xz^3 +z^5
+y^2zt+(xy+yz^2)t^2+(xz+z^3)t^3+yt^5+zt^6
$$
Therefore,  the edge $P_0P_3$ given by $(y=z=0)$ is contained on $S_T$.  However, this edge is a line of $A_1$ singularities  so the surface $S_T$ is not normal. Note that the other singularities on $S_T$ are dual to the ones on $\mathcal{B}_{Q_{12}} $
 $$
\mathcal{\hat{B}}_{Q_{12}} 
=
\left\{ 
2 \times \frac{1}{3}(1,1), \frac{1}{6}(1,4)
\right\}
$$
\end{rmk}

\subsection{ The $S_{11}$ singularity} It is also known as $ D_{2,5,6}$ or $Ta(-3,-4)$. Its normal form is  $x^2z+y^2x+z^4$, and its unipotent base change $t^{16} \to t$ induces the 
$58$th case on Yonemura classification.  The set of singularities are:
\begin{align*}
\mathcal{B}_{S_{11}} 
=
\left\{ 
\frac{1}{2}(1,1),  \frac{1}{5}(1,1), \frac{1}{6}(1,1)
\right\}
& & 
\mathcal{\hat{B}}_{S_{11}} 
=
\left\{ 
\frac{1}{2}(1,1),  \frac{1}{5}(1,4), \frac{1}{6}(1,5)
\right\}
\end{align*} 
The associated K3 surface is given by $S_{16} \subset \mathbb{P}(6,5,4,1)$

\subsection{The singularity $S_{12}$} It is also known as $D_{3,4,5}$ or  $Tr(-2,-3,-4)$. Its normal form  is $x^2z+xy^2+yz^3$ and its unipotent base change $t^{13} \to t$ induces the 87th case on Yonemura classification.  The set of singularities are:
\begin{align*}
\mathcal{B}_{S_{12}} 
=
\left\{ 
 \frac{1}{3}(1,1), \frac{1}{4}(1,1), \frac{1}{5}(1,1)
\right\}
& & 
\mathcal{\hat{B}}_{S_{12}} 
=
\left\{ 
\frac{1}{3}(1,2),  \frac{1}{4}(1,3), \frac{1}{5}(1,4)
\right\}
\end{align*} 
The associated K3 surface is $S_{13} \subset \mathbb{P}(5,4,3,1)$
%A smoothing set of well formed and quasismooth weights: (5, 4, 3, 2) is not normal

\subsection{The singularity $Z_{11}$} It is also known as  $D_{2,3,8}$ or $Cu(-2)$. Its normal form is  $x^2+y^3z+z^5$  and its unipotent base change $t^{30} \to t$ induces the 38th case on the Yonemura smoothing.  The set of singularities are:
\begin{align*}
\mathcal{B}_{Z_{11}} 
=
\left\{ 
\frac{1}{2}(1,1), \frac{1}{3}(1,1), \frac{1}{8}(1,1)
\right\}
& & 
\mathcal{\hat{B}}_{Z_{11}} 
=
\left\{ 
\frac{1}{2}(1,1),  \frac{1}{3}(1,2), \frac{1}{8}(1,7)
\right\}
\end{align*} 
The associated K3 surface is $S_{30}  \subset \mathbb{P}(15,8,6,1)$ 

\subsection{The singularity $Z_{12}$} It is also known as  $D_{2,4,6}$ or $ Ta(-2,-4)$. Its normal form is $ x^2+y^3z+yz^4$, and its unipotent base change $t^{22} \to t$ induces the 78th case of the Yonemura classification.   The set of singularities are:
\begin{align*}
\mathcal{B}_{Z_{12}} 
=
\left\{ 
\frac{1}{2}(1,1), \frac{1}{4}(1,1), \frac{1}{6}(1,1)
\right\}
& & 
\mathcal{\hat{B}}_{Z_{12}} 
=
\left\{ 
\frac{1}{2}(1,1),  \frac{1}{4}(1,3), \frac{1}{6}(1,5)
\right\}
\end{align*} 
The associated K3 surface is  $S_{22} \subset  \mathbb{P}(11,6,4,1)$ 

\subsection{The singularity $Z_{13}$} It is also known as $D_{3,3,5}$ or  $Tr(-2,-2,-4)$. Its normal form is $ x^2+y^3z+z^6$  and its unipotent base change $t^{18} \to t$ induces the 39th case on the Yonemura's classification.
\begin{align*}
\mathcal{B}_{Z_{13}} 
=
\left\{ 
2 \times \frac{1}{3}(1,1), \frac{1}{5}(1,1)
\right\}
& & 
\mathcal{\hat{B}}_{Z_{13}} 
=
\left\{ 
2 \times \frac{1}{3}(1,2),  \frac{1}{5}(1,4)
\right\}
\end{align*} 
The associated K3 surface is  $S_{18} \subset \mathbb{P}(9,5,3,1)$,

\subsection{The line bundle of the exceptional surface} The following lemmas are used to prove the ampleness of the line bundle on Theorem \ref{thm2}.
\begin{lemma}\label{C2}
Let $\tilde{S}_T$ be the smooth model of $S_T$, let $E_k$ be the exceptional divisors associated to the resolution $\varphi:\tilde{S}_T  \to S_T$, and let $\tilde{C}$ be the proper transform of $C =S_1\big|_{ S_T}$  on $\tilde{S}_T$. Then, it holds:
$$
C^2=(\varphi^*(C))^2 = \tilde{C}^2-\sum_{j,k}(E_j.\tilde{C})D_{jk}(E_k.\tilde{C}) 
$$
where $D_{jk}$ is the inverse of the intersection matrix $E_jE_k$. In particular, if all the singularities on $\mathcal{\hat{B}}_{S_0}$ are of type
$\{ A_{k_1}, \ldots A_{k_m} \}$, then 
\begin{align}\label{sumAk}
C^2= \tilde{C}^2 + \sum_{j=1}^{j=m} \frac{k_j}{k_j+1}
\end{align}
\end{lemma}
\begin{proof}
By the projection formula, we have  $\varphi^*(C).E_k=0$ for every exceptional divisor $E_k$. This implies 
$$
C^2=\varphi^*(C).\left( \tilde{C}+\sum_j a_jE_j \right)= \tilde{C}^2+\sum_ja_j(E_j.\tilde{C})
$$
On other hand $\varphi^*(C).E_k=0$ implies
$ \sum_{j}a_jE_j.E_k =-\tilde{C}.E_k$ where $E_j.E_k$ is the intersection matrix. Then
$$
a_j=\sum_k D_{jk}(-\tilde{C}.E_k)
$$
where $D_{jk}$ is the inverse of the intersection matrix $E_jE_k$. The previous expressions imply
$$
C^2=\tilde{C}^2+\sum_j \left( \sum_k D_{jk}(-\tilde{C}.E_k)    \right)E_j.\tilde{C}
=
\tilde{C}^2-\sum_{j,k}(E_j.\tilde{C})D_{jk}(E_k.\tilde{C}) 
$$
Let $T_i$ be the cyclic quotient singularities of $S_T$ supported on the curve $C$; and 
let $E(T_i)$ be the intersection matrix of each singularities $T_i$. Then, 
\begin{align*}
D_{j,k} =
 \begin{pmatrix}
E(T_1)^{-1} &   \cdots           & 0 \\
  0                 & E(T_2)^{-1}     & 0  \\
  \vdots      & \vdots              & \vdots  \\
                 0 & \cdots              & E(T_m)^{-1}
 \end{pmatrix}
\end{align*}
In particular, the exceptional divisors associated to each $T_i$ do not intersect. Therefore, we can consider the contribution of each singularity $T_i$ independently:
$$
\sum_{j,k}(E_j.\tilde{C})D_{jk}(E_k.\tilde{C}) 
=
\sum_{E_j , E_k\in R_1}(E_j.\tilde{C})(E(T_1)^{-1})_{jk}(E_k.\tilde{C}) 
+
\ldots
+
\sum_{E_j , E_k\in R_m}(E_j.\tilde{C})(E(T_m)^{-1})_{jk}(E_k.\tilde{C}) 
$$  
where $E_j \in R_i$ means that $E_j$ is an exceptional divisor associated to $T_i$. 
Only the first exceptional divisor of the resolution of $T_i$ intersects the curve $\tilde{C}$ at a point. Therefore,
$$
\sum_{E_j , E_k\in R_i}(E_j.\tilde{C})(E(T_i)^{-1})_{jk}(E_k.\tilde{C}) 
=
\left( E(T_i)^{-1} \right)_{1,1}
$$
If $T_i$ is an $A_k$ singularity, then by the configuration of its exceptional curves and by writing $D_{jk}$ in terms of the matrix of cofactors, we have: 
$$
(E(A_k)^{-1})_{1,1} = -\frac{k}{k+1}. 
$$  
From which our statement follows.
\end{proof}

The following is a well known result on degeneration of surfaces.
\begin{lemma}\label{tpf}
Let $Y \to \Delta$ be an one dimensional degeneration of surfaces such that $Y_t$ is smooth and  $Y_0 = \sum_i n_iV_i$.  Denote by $C_{ij} $ the double curve $V_i|_{V_j} \subset V_i$, and the triple point intersection 
$
T_{ijk} = V_i \cap V_j \cap V_k
$. Then, we have
\begin{align*}
N^{n_i}_{V_i | Y}
&=
\mathcal{O}_{V_i} \left(  - \sum_{j \neq i} n_j C_{ij} \right)
\\
V_i^2V_j 
&= C_{ji}^2
\\
n_jC_{ij}^2+n_iC_{ji}^2 
&=- \sum_{k \neq i,j} n_k | T_{ijk} |
\end{align*}
\end{lemma}
On the Expression \ref{sumAk}, we need to find the value of $\tilde{C}^2$. That is the purpose of the following result.
\begin{lemma}\label{tC2}
Let $Y$ be a smooth model of the unipotent degeneration $ X \to \Delta$, so its central fiber has reduced components $Y|_0 = \sum V_i$.  Let $\tilde{S}_T$ be the proper transform of $S_T$ on $Y|_0$, and  let $\tilde{C}$ be the proper transform of $C=S_1 \big|_{S_T}$ on $\tilde{S}_T$.  Consider the set of singularities 
$$
B_{S_0} = \left\{ \frac{1}{\alpha_k}(1, \beta_k)  \right\}
$$
 Then, it holds
$$
\tilde{C}^2=\frac{\deg_{\tilde{w}}(S_0)}{w_0w_1w_2} + \sum_{k} \frac{\beta_k}{\alpha_k} - n_p
$$
where $n_p$ is the number of cyclic quotient singularities supported on $C$. In particular, for  Fuchsian hypersurface singularities we have:
$$
\tilde{C}^2=\frac{\deg_{\tilde{w}}(S_0)}{w_0w_1w_2} + \sum_{k} \frac{1}{k+1} - n_p
$$
\end{lemma}
\begin{proof}
Let $\tilde{S}_1$ be the proper transform of $S_1$ on $Y \big|_0$ then it holds that 
$\tilde{S}_1 \cap \tilde{S}_T$ support $n_p$ triple points where the extra surfaces are exceptional divisors of the cyclic quotient singularities supported on $C$. 
Let $\hat{C}$ be the proper transform of $C$ on  $\tilde{S}_1$, by Lemma \ref{tpf}, it holds that:
$$
\tilde{C}^2+\hat{C}^2
=- \sum_{k \neq i,j} | T_{ijk} |  = -n_p
$$
By \cite[Thm 3.6.1]{orlik1971isolated} we have that 
$$
- \hat{C}^2 =  \frac{ \deg_{\mathbf{w}}(f)}{w_0w_1w_2} + \sum_k \frac{\beta_k}{\alpha_k}
$$
so the statement follows from it. In the case of Fuchsian hypersurface singularities $\beta_k=1$ (see Example \ref{fsh}). 
\end{proof}

%%%%%%%%%%%%%%%%%%%%%%%%%%%%%%%%%%%%%%
\section{Higher modality singularities}\label{hm}

Next, we discuss several examples  of singularities which behaviour is different to the unimodal ones.  We draw our examples from minimal elliptic singularities.

\subsection{ The $V_{18}'$ singularity} It is also known as  $4A_{1,-2,o}$; this is a minimal elliptic singularity with Milnor number $18$ and modality $4$.  Its normal form is $x^3+y^4+z^4$, and its set of cyclic quotient singularities is 
$\mathcal{B}_{V_{18}'} = 4A_2 $. There are two dual sets such that can be realized by a weighted blow up:
\begin{align*}
\mathcal{\hat{B}}^1_{V_{18}'} = 4 \times \frac{1}{3}(1,1) 
& &
\mathcal{\hat{B}}^2_{V_{18}'} = 4 \times \frac{1}{3}(1,2)
\end{align*}
By Theorem \ref{thm1}, the  smoothing associated to the dual set $\mathcal{\hat{B}}^1_{V_{18}'}$ is  realized by a linear combination of the monomials of weight $12$ with respect the weights 
$\textbf{v} = (4,3,3,1)$.  This is the unipotent smoothing, the threefold $\tilde{X}$ has singularities $\frac{1}{3}(1,1,2)$, and the exceptional surface is $S_{12} \subset \mathbb{P}(4,3,3,1)$.   
\begin{comment}
Using the notation of Lemmas \ref{C2} and \ref{tC2}, we have that $n_p=4$, $-\hat{C}^2 =3$, and $\tilde{C}^2=-1$.   The matrix $D_{jk} = diag\left( -1/3, \ldots , -1/3 \right)$  implies that
$$
\left(S_1 \big|_{S_T} \right)^2 = 1/3
$$
In this case $K_{\tilde{X}} \big|_{S_T} = K_{S_T}+S_1 \big|_{S_T}$, but 
$K_{S_T} \neq \mathcal{O}_{S_T}$ because $S_T$ is not a normal K3 surface
(see the classification at \cite[II.3.3]{wps}). In fact  $h^0(S_T, K_{S_T}) = 1$.
\end{comment}

On other hand, the smoothing associated to the dual set  $\mathcal{\hat{B}}^2_{V_{18}'}$ is realized by a linear combination of the monomials of degree 12 with respect the weights $\textbf{u}=(4,3,3,2)$. This threefold has a strictly log canonical singularity, and it corresponds to 2nd case on Yonemura's classification. The exceptional surface is the normal $K3$ surface $S_{12} \subset \mathbb{P}(4,3,3,2)$ with singularities   $4A_2 + 3A_1 $.

\subsection{The   $E_{20}$  singularity.} \label{e20}
It is also known as $E_{8,-3}$. It is a minimal elliptic singularity with Milnor number $20$ and modality $2$.  Its normal form is $x^2+y^3+z^{11}$ and its set of cyclic quotient singularities is 
$$
\mathcal{B}_{E_{20}}=\left\{
\frac{1}{2}(1,1), \frac{1}{3}(1,2), \frac{1}{11}(1,9)
\right\}
$$
This set has four sets of dual singularities.  However, only  two of them have an associated smoothing (see Example \ref{notunq}).
\begin{enumerate}
\item The dual set:
$$
\mathcal{\hat{B}}^1_{E_{20}}=\left\{
\frac{1}{2}(1,1), \frac{1}{3}(1,1), \frac{1}{11}(1,10)
\right\}
$$
which associated smoothing $X_1 \to \Delta$ is realized by a linear combination of the monomials of degree $66$ with respect  the weights $\textbf{v}=(33,22,6,5)$. The threefold $X_1 \subset  \mathbb{C}^4$ has a strictly log canonical singularity and it correspond to 46th case  of Yonemura's classification. The exceptional surface is the normal K3 surface 
$
S_{66} :=
(x^2+y^3+z^{11}+t^{12}z=0) 
 \subset \mathbb{P}(33,22,11,5)
$
\item The smoothing associated to the dual set
$$
\mathcal{\hat{B}}_{E_{20}}=\left\{
\frac{1}{2}(1,1), \frac{1}{3}(1,1), \frac{1}{11}(1,2)
\right\}
$$
is the unipotent one. In that case, the surface
$
S_{66} \subset \mathbb{P}(33,22,6,1) 
$
is an non minimal K3 surface (see also \cite[Sec 3.3.3]{hunt1999k3})
\end{enumerate}

\bibliographystyle{plain}
\bibliography{nsfRef.bib} 

\begin{thebibliography}{10}

\bibitem{arnold}
V.I. Arnold, S.M. Guseuin~Zade, and A.N. Varchenko.
\newblock {\em Singularities of Differentiable Maps: The classification of
  critical points, caustics and wave fronts}, volume~1.
\newblock Birkhauser, 1985.

\bibitem{dolgachev1974quotient}
Igor~V Dolgachev.
\newblock Quotient-conical singularities on complex surfaces.
\newblock {\em Functional Analysis and Its Applications}, 8(2):160--161, 1974.

\bibitem{dolgachev1996mirror}
Igor~V Dolgachev.
\newblock Mirror symmetry for lattice polarizedk3 surfaces.
\newblock {\em Journal of Mathematical Sciences}, 81(3):2599--2630, 1996.

\bibitem{ebeling2002poincare}
Wolfgang Ebeling.
\newblock Poincar{\'e} series and monodromy of a two-dimensional
  quasihomogeneous hypersurface singularity.
\newblock {\em manuscripta mathematica}, 107(3):271--282, 2002.

\bibitem{ebeling2003poincare}
Wolfgang Ebeling.
\newblock The poincar{\'e} series of some special quasihomogeneous surface
  singularities.
\newblock {\em Publications of the Research Institute for Mathematical
  Sciences}, 39(2):393--413, 2003.

\bibitem{wps}
A.R. Fletcher.
\newblock {\em Working with weighted complete intersections}.
\newblock Max-Planck-Inst. f. Mathematik, 1989.

\bibitem{fujiki}
Akira Fujiki.
\newblock On resolutions of cyclic quotient singularities.
\newblock {\em Publications of the Research Institute for Mathematical
  Sciences}, 10(1):293--328, 1974.

\bibitem{gallardo2013git}
Patricio Gallardo.
\newblock On the {GIT} quotient of quintic surfaces.
\newblock {\em arXiv preprint arXiv:1310.3534}, 2013.

\bibitem{hassett}
B.~Hassett.
\newblock Local stable reduction of plane curve singularities.
\newblock {\em Journal fur die Reine und Angewandte Mathematik}, 520:169--194,
  2000.

\bibitem{hunt1999k3}
Bruce Hunt and Rolf Schimmrigk.
\newblock K3-fibered calabi--yau threefolds i: The twist map.
\newblock {\em International Journal of Mathematics}, 10(07):833--869, 1999.

\bibitem{ishiicw}
S.~Ishii.
\newblock The canonical modifications by weighted blow-ups.
\newblock {\em J. Algebraic Geom.}, 5(4):783--799, 1996.

\bibitem{ishiimodels}
S.~Ishii.
\newblock Minimal, canonical and log-canonical models of hypersurface
  singularities.
\newblock In {\em Birational algebraic geometry ({B}altimore, {MD}, 1996)},
  volume 207 of {\em Contemp. Math.}, pages 63--77. Amer. Math. Soc.,
  Providence, RI, 1997.

\bibitem{geoK3}
S.~Ishii and K.~Watanabe.
\newblock A geometric characterization of a simple {$ K3 $}-singularity.
\newblock {\em Tohoku Mathematical Journal}, 44(1):19--24, 1992.

\bibitem{smmp}
J{\'a}nos Koll{\'a}r.
\newblock {\em Singularities of the minimal model program}, volume 200.
\newblock Cambridge University Press, 2013.

\bibitem{KSB88}
J{\'a}nos Koll{\'a}r and Nicholas~I Shepherd-Barron.
\newblock Threefolds and deformations of surface singularities.
\newblock {\em Inventiones mathematicae}, 91(2):299--338, 1988.

\bibitem{lme}
Henry~B Laufer.
\newblock On minimally elliptic singularities.
\newblock {\em American Journal of Mathematics}, 99(6):1257--1295, 1977.

\bibitem{morrison1984clemens}
David~R Morrison.
\newblock The {C}lemens-{S}chmid exact sequence and applications.
\newblock {\em Topics in transcendental algebraic geometry (Ph. Griffiths,
  ed.)}, pages 101--119, 1984.

\bibitem{morrison1985canonical}
David~R Morrison.
\newblock Canonical quotient singularities in dimension three.
\newblock {\em Proceedings of the American Mathematical Society},
  93(3):393--396, 1985.

\bibitem{morrison1984terminal}
David~R Morrison and Glenn Stevens.
\newblock Terminal quotient singularities in dimensions three and four.
\newblock {\em Proceedings of the American Mathematical Society}, 90(1):15--20,
  1984.

\bibitem{oka}
M.~Oka.
\newblock On the resolution of the hypersurface singularities.
\newblock {\em Complex analytic singularities}, 8:405--436, 1986.

\bibitem{orlik1971isolated}
P.~Orlik and P.~Wagreich.
\newblock Isolated singularities of algebraic surfaces with
  {$\mathbb{C}^*$}-action.
\newblock {\em The Annals of Mathematics}, 93(2):205--228, 1971.

\bibitem{pinkham1978deformations}
Henry Pinkham.
\newblock Deformations of normal surface singularities with {$\mathbb{C}^*$}
  action.
\newblock {\em Mathematische Annalen}, 232(1):65--84, 1978.

\bibitem{prokhorov}
Y.G Prokhorov.
\newblock Elliptic gorenstein singularities, log canonical thresholds, and log
  enriques surfaces.
\newblock {\em Journal of Mathematical Sciences}, 115(3):2378--2394, 2003.

\bibitem{Rana}
J.~Rana.
\newblock {\em Boundary divisors in the moduli space of stable quintic
  surfaces}.
\newblock PhD thesis, UMass Amherst University, Amherst, Massachusett, December
  2013.

\bibitem{sage}
W.\thinspace{}A. Stein et~al.
\newblock {\em {S}age {M}athematics {S}oftware ({V}ersion x.y.z)}.
\newblock The Sage Development Team, 2013.
\newblock {\tt http://www.sagemath.org}.

\bibitem{wahl1988deformations}
J.~Wahl.
\newblock Deformations of quasi-homogeneous surface singularities.
\newblock {\em Mathematische Annalen}, 280(1):105--128, 1988.

\bibitem{wahl2011log}
J.~Wahl.
\newblock Log-terminal smoothings of graded normal surface singularities.
\newblock {\em arXiv preprint arXiv:1110.2979}, 2011.

\bibitem{watanabe1998three}
K.~Watanabe.
\newblock Three dimensional hypersurface purely elliptic singularities.
\newblock {\em RIMS Kokyuroku}, 1037:69--73, 1998.

\bibitem{xu1989classification}
Y.~Xu and S-T Yau.
\newblock Classification of topological types of isolated quasi-homogeneous two
  dimensional hypersurface singularities.
\newblock {\em manuscripta mathematica}, 64(4):445--469, 1989.

\bibitem{yonemura}
T.~Yonemura.
\newblock Hypersurface simple {$K3$} singularities.
\newblock {\em Tohoku Mathematical Journal}, 42(3):351--380, 1990.

\end{thebibliography}
\end{document}